\documentclass[preprintnumbers,amsmath,11pt,amssymb,floatfix,superscriptaddress,nofootinbib]{article}
\usepackage{amsmath,amssymb,latexsym,textcomp,mathrsfs}
\usepackage[all]{xy}
\usepackage{graphicx}
\usepackage{bm,amsmath, amsthm, amssymb, amsfonts}
\usepackage{times}
\usepackage[english]{babel}
\usepackage{setspace}
\setbox0=\hbox{$+$}
\newdimen\plusheight
\plusheight=\ht0
\def\+{\;\lower\plusheight\hbox{$+$}\;}

\setbox0=\hbox{$-$}
\newdimen\minusheight
\minusheight=\ht0
\def\-{\;\lower\minusheight\hbox{$-$}\;}

\setbox0=\hbox{$\cdots$}
\newdimen\cdotsheight
\cdotsheight=\plusheight
\def\cds{\lower\cdotsheight\hbox{$\cdots$}}

\numberwithin{equation}{section}
\theoremstyle{plain}
\newtheorem{theorem}{Theorem}[section]
\newtheorem{lemma}{Lemma}[section]

\newtheorem{example}{Example}[section]
\newtheorem{definition}{Definition}[section]

\newtheorem{proposition}{Proposition}[section]
\newtheorem{remark}{Remark}[section]

\newtheorem{note}{Note}[section]

\def\mytitle#1{\setcounter{equation}{0}
\setcounter{footnote}{0}
\begin{flushleft}\Large\textbf{#1}\end{flushleft}
\vspace{0.20cm}}
\def\myname#1{\leftline{{\large #1}}\vspace{-0.13cm}}
\def\myplace#1#2{\small\begin{flushleft}\textit{#1}\\
\texttt{#2}\end{flushleft}}

\def\myclassification#1{\small\noindent
Keywords : Strong-$I^K$-Convergence, Strong-$(I\vee K)^K$-Convergence, Strong-$I^K$-Cauchy, Strong-$I^K$-Limit Points.
       #1\vspace{0.5cm}\\
			AMS subject classification$(2010)$: Primary: 54E70 ; Secondary: 40A35}
\usepackage{graphicx}
\usepackage{amsmath}

\begin{document}
\mytitle{Strong-$I^K$-Convergence in Probabilistic Metric Spaces}

\myname{$Amar Kumar Banerjee^{\dag}$\footnote{akbanerjee@math.buruniv.ac.in;akbanerjee1971@gmail.com} and  $Mahendranath~ Paul^{\dag}$\footnote{mahendrabktpp@gmail.com}}
\myplace{$\dag$Department of Mathematics, The University of Burdwan, Purba Burdwan -713104, India.} {}
\begin{abstract}
In this paper we study the idea of strong-$I^K$-convergence of functions which is common generalization of strong-$I^*$-convergence of functions in probabilistic metric spaces. We also study strong-$I^{K}$-limit points of functions in the same space.
\end{abstract}
\myclassification{}
\section{Introduction}
The work of generalization of convergence of sequences were taken into consideration in the early sixties of twentieth century. The idea of usual convergence of a real sequence was extended to statistical convergence by H. Fast \cite{10} and then H. Steinhaus \cite{33} in the year $1951$ and later it was developed by several authors\cite{1,12,31,32}. Now we recall natural density of a set.
Let $\mathbb{N}$ denotes the set of natural numbers. If $K\subset \mathbb{N}$, then $K_n$ will denote the set $\{k \in K:k\leq n$\} and $|K_n|$ stands for the cardinality of $K_n$. The natural density of $K$ is then defined by
$$d(K)=\displaystyle{\lim_{n}}\frac{|K_n|}{n}$$
if the limit exits. A real sequence $\{x_n\}$ is said to be statistically convergent to $l$ if for every $\epsilon>0$ the set 
$K(\epsilon)=\{k\in\mathbb{N}:|x_k-l|\geq \epsilon\}$ has natural density zero\cite{10,20}. In the year 2000, the concept of $\lambda$-statistical convergence was introduced by Mursaleen in\cite{27} as an extension of statistical convergence. The another generalization of statistical convergence is the idea of ideal convergence(i.e. $I$ and $I^*$-convergence) which depends on the structure of ideals of subsets of the natural numbers introduced by P.Kostyrko et al.\cite{19}  in the beginning of twenty first century. $I$-convergence of real sequences coincides with the ordinary convergence if $I$ is the ideal of all finite subsets of $\mathbb{N}$ and with the statistically convergence if $I$ is the ideal of $\mathbb{N}$ of natural density zero\cite{17,19}.\\
The concept of $I^*$-convergence which is closely related to that of $I$-convergence was introduced  by P.Kostyrko et al.\cite{19}.  Subsequently the idea of $I$-convergence has been extended from the real number space to the metric spaces and the normed linear spaces by many authors. In $2005$, B.K.Lahiri and Pratulananda Das \cite{17} extended the concept of $I$ and $I^*$-convergence in a topological space and they observed that the basic properties are preserved also in a topological space. Later many works on $I$-convergence were done in topological spaces\cite{2,3,3.1,3.3,3.4,3.5}.\\
Ordinary convergence always implies statistical convergence and when $I$ is an admissible ideal, $I^*$-convergence implies  $I$-convergence. But converse may not be true. Several examples were given for real sequences in\cite{19} to disprove the converse part. Moreover in \cite{6,17,19} it is seen that a statistical convergent sequence and $I$ and $I^*$- convergent sequence  even need not be bounded.\\
In the year 2010, M. Macaj and M. Sleziak \cite{23} introduced the idea of  $I^K$-convergence in a topological space where $I$ and $K$ are ideal of an arbitrary set $S$ and shown that this type of convergence is a common generalization for all types of $I$ and $I^*$-convergence in some restriction. They also gave the AP$(I,K)$ condition which is generalization of AP condition given in \cite{19}.\\
The concept of $I$-Cauchy condition was studied first by K. Dems \cite{4} in 2004 and then further investigation on $I^*$-Cauchy was studied in \cite{29} by A. Nabiev et al in 2007. In the year 2014, P. Das et al \cite{7} studied on $I^K$-Cauchy functions. \\
The idea of probabilistic metric space was first introduce by Menger \cite{24} as a generalization of ordinary metric space. The notion of distance has a probabilistic nature which has led to a remarkable development of the probabilistic metric space(in short PM Space). PM Spaces have nice topological properties and several topologies can be defined on this space and the topology that was found to be most useful is the strong topology. The theory was brought to its present form by Schweizer and Sklar \cite{36}, Tardiff \cite{39}. In the year 2009 , the concept of statistical convergence and then strong ideal convergence in probabilistic metric space was studied in \cite{37,38} by C.Sencimen et al. In the year 2012, M. Mursaleen et al studied ideal convergence in probabilistic normed spaces.\cite{28}\\
The recent works of generalization of convergence via ideals in probabilistic metric space have been developed by many authors. It seems therefore reasonable to think if we extend the same in the same space using double ideals and in that case we intend to investigate how far several the basic properties (such as results on limit points, Cauchy sequences etc.) are affected. In our paper we study the idea of strong-$I^K$-convergence of functions in a probabilistic metric space which also generalizes the strong-$I^*$-convergence in \cite{38}. Since the convergence in PM space is very significant to probabilistic analysis, we realize that the idea of convergence via double ideal in a PM space would give more general frame for analysis of PM space. 
\section{Preliminaries}
First we focus on some basic idea related to theory of PM spaces which are already studied in depth in the book by Schweizer and Sklar \cite{35}.
\begin{definition}
A non-decreasing function $F:\mathbb{R}\rightarrow [0,1]$ with $F(-\infty)=0$ and $F(\infty)=1$ is called a distribution function.
\end{definition}
The well-known notation $\bigtriangleup$ denotes the set of all left continuous distribution function.
\begin{definition}
A non-decreasing function $F$ defined on $[0,\infty]$ which satisfies $F(0)=0$, $F(\infty)=1$ and left continuous on $(0,\infty)$ is called a distance distribution function (d.d.f). 
\end{definition}
The set of all distance distribution functions denoted by $\triangle^+$. In $1942$, K. Menger, who had played a main role in the development of the theory of metric spaces, proposed a probabilistic generalization of this theory. He proposed replacing the number $d(a,b)$ by a real function $F_{ab}$ whose value $F_{ab}(s)$, for any real no $s$, is interpreted as the probability that the distance between $a$ and $b$ is less than $s$. Then $0\leq F_{ab}(s)\leq 1$, for every real value $s$ and clearly $F_{ab}(s)\leq F_{ab}(t)$ whenever $s\leq t$. Hence $F_{ab}$ is a probabilistic distribution function.
\begin{definition}
For any $x\in(-\infty,\infty)$ the unit step function at $x$ is denoted by $\epsilon_x$ and is defined to be a function in the family of distribution functions given by 
\[\epsilon_x(s)=\left\{\begin {array}{ll}
        0 & \mbox{if $s\leq x$} \\
		1 & \mbox{if $s> x$}
		\end{array}
		\right. \]
\end{definition}
So in particular 
\[\epsilon_0(s)=\left\{\begin {array}{ll}
        0 & \mbox{if $s\leq 0$} \\
		1 & \mbox{if $s> 0$}
		\end{array}
		\right. \]
\begin{definition}
The distance between $F$ and $G$ in $\bigtriangleup$ is defined by\\ $d_L(F,G)=inf\{t\in(0,1]:~both~(F,G;t)~and~(G,F;t)~hold\}$ where for $t\in(0,1]$, the condition $(F,G;t)$ holds if $F(s-t)-t\leq G(s)\leq F(s+t)+t$ for every $s\in(-\frac{1}{t},\frac{1}{t})$.
\end{definition}

\begin{theorem}
 $(\triangle,d_L)$ is a metric space which is compact and hence complete.
\end{theorem}
\begin{definition}
A sequence $\{F_n\}_{n\in \mathbb{N}}$ of d.d.f's is said to converge weakly to a d.d.f $F$ and if $\{F_n(s)\}_{n\in \mathbb{N}}$ converges to $F(s)$ at each continuity point $s$ of $F$ and then we write $F_n\xrightarrow{w} F$.
\end{definition}
In order to present the definition of a probabilistic metric space, we need the notion of triangle function introduced by Serstnev in \cite{38.1}.
\begin{definition}
A triangle function $\tau:\triangle^+\times\triangle^+\rightarrow\triangle^+$ is a binary operation on $\triangle^+$ which is non-decreasing, associative, commutative  in each of its variables and has $\epsilon_0$ as the identity.
\end{definition}
\begin{definition}
A probabilistic metric space (briefly PM space) is a triplet $(P,\mathcal{F},\tau)$ where $P$ is a non-empty set, $\mathcal{F}:P\times P\rightarrow \triangle^+$ is a function, $\tau$ is a triangle function satisfying the following condition for all $a,b,c\in P$\\
(i) $\mathcal{F}(a,a)=\epsilon_0$\\
(ii)$\mathcal{F}(a,b)\neq \epsilon_0$ if $a\neq b$\\
(iii)$\mathcal{F}(a,b)=\mathcal{F}(b,a)$\\
(iv)$\mathcal{F}(a,c)\geq\tau(\mathcal{F}(a,b),\mathcal{F}(b,c))$
\end{definition}
Henceforth we shall denote $\mathcal{F}(a,b)$ by $F_{ab}$ and its value at $s$ by $F_{ab}(s)$.
\begin{theorem}\label{1.0}
Let $G\in \triangle^+$ be given then for any $t>0$, $G(t)>1-t$ if and only if $d_L(G,\epsilon_0)<t$
\end{theorem}
\begin{definition}
Let $(P,\mathcal{F},\tau)$ be a PM space. For $t>0$ and $a\in P$, the strong-$t$-neighborhood of $a\in P$ is defined by the set $\mathcal{N}_a(t)=\{b\in P:F_{ab}(t)>1-t\}$.
\end{definition}
The collection $\aleph_a=\{\mathcal{N}_a(t):t>0\}$ is called strong neighborhood system at $a$ and the union $\aleph=\cup_{a\in P}\aleph_a$ is said to be strong neighborhood system of $S$ and the strong topology is introduced by a strong neighborhood system.
Applying \ref{1.0} we can write strong-$t$-neighborhood as $\mathcal{N}_a(t)=\{b\in P:d_L(F_{ab},\epsilon_0)<t\}$.
\begin{theorem}
Let $(P,\mathcal{F},\tau)$be a PM space. If $\tau$ is continuous, then the strong neighborhood system $\aleph$ satisfies $(i)$ and $(ii)$.\\
$(i)$ If $V$ is a strong neighborhood of $p\in P$ and $q\in V$ , then there is a strong neighborhood $W$ of $q$ such that $W\subseteq V$.\\
$(ii)$ If $p\neq q$, then there is a $V \in \mathcal{N}_p$ and a $W$ in $\mathcal{N}_q$ such that $V\cap W =\phi$ and thus the strong neighborhood system $\mathcal{N}$ determines a Hausdorff topology for $P$.
\end{theorem}
\begin{definition}
Let $(P,\mathcal{F},\tau)$ be PM space. Then for any $t>0$, the subset $\mathscr{U}(t)$ of $P\times P$ given by $\mathscr{U}(t)=\{(a,b):F_{ab}(t)>1-t\}$ is called strong-$t$-vicinity.
\end{definition}
\begin{theorem}\label{2}
Let $(P,\mathcal{F},\tau)$ be PM space and $\tau$ be continuous. Then for any $t>0$, there is an $\eta>0$ such that $\mathscr{U}(\eta)\circ\mathscr{U}(\eta)\subseteq\mathscr{U}(t)$, where $\mathscr{U}(\eta)\circ\mathscr{U}(\eta)=\{(a,c):for~some~ b,(a,b)~and~(b,c)\in \mathscr{U}(\eta)\}$
\end{theorem}
\begin{note}
Under the hypothesis of theorem \ref{2} we can say that for any $t>0$ there is an $\eta>0$ such that $F_{ac}(t)>1-t$ whenever $F_{ab}(\eta)>1-\eta$ and $F_{bc}(\eta)>1-\eta$ i.e. from the theorem \ref{1.0} we can say $d_L(F_{ac},\epsilon_0)<t$ whenever $d_L(F_{ab},\epsilon_0)<\eta$ and $d_L(F_{bc},\epsilon_0)<\eta$.
\end{note}
\begin{definition}
Let $S$ be a non-void set then a family of sets $I\subset 2^S$ is said to be an ideal if 
\item (i) $A,B\in I \Rightarrow A\cup B\in I$
\item(ii) $A\in I, B\subset A \Rightarrow B\in I$
\end{definition}
$I$ is called nontrivial ideal if $S\notin I$ and $ I\neq \{\phi\} $. In view of condition (ii) $\phi\in I $. If $I \subsetneqq 2^S$ we say that $I$ is proper ideal on $S$. Several examples of non-trivial ideals are seen in \cite{19}.
A nontrivial ideal $I$ is called admissible if it contains all the singleton of $S$. A nontrivial ideal $I$ is called non-admissible if it is not admissible.
The ideal of all finite subsets of $S$ which we shall denote by Fin$(S)$. If $S=\mathbb{N}$, set of all natural number, then we write Fin instead of Fin$(\mathbb{N})$ for short.
\begin{note}
The dual notion to the ideal is the notion of the filter i.e. a filter on $S$ is non-void system of subsets of $S$, which is closed under finite intersection and super sets. If $I$ is a non-trivial ideal on $S$ then $F=F(I)=\{A\subset S:S\setminus A \in I \}$ is clearly a filter on $S$ and conversely. $F(I)$ is called associated filter with respect to ideal $I$.
\end{note}
\section{Strong-$I^K$-Convergence of Functions}
Throughout the paper $P$ stands for a probabilistic metric space(briefly PM space) and we always assume that in a PM space $P$, the triangle function $\tau$ is continuous and $P$ endowed with strong topology. $I$, $K$ are non-trivial ideals of a non empty set $S$ unless otherwise stated.
First we will give the definition of Fin-convergence of a function in PM space
\begin{definition}
Let $(P,\mathcal{F},\tau)$ be a PM space. A function $f:S\rightarrow P$ is said to be Fin-convergent to $p\in P$ if
$f^{-1}(\mathcal{N}_p(t))=\{s\in S:f(s)\in \mathcal{N}_p(t)\}$ is a finite set for every strong-t-neighborhood $\mathcal{N}_p(t)$ of $p$.
\end{definition}
We use the notation Fin$(S)$-$f= p$.
Now we give the definition of strong-$I$-convergence using function instead of sequence in probabilistic metric space.
\begin{definition}(cf.\cite{38}) Let $I$ be an ideal on a non-empty set $S$ and $(P,\mathcal{F},\tau)$ be a PM space. A function $f:S\rightarrow P$ is said to be strong-$I$-convergent to $p\in P$ if
$$f^{-1}(\mathcal{N}_p(t))=\{s\in S:f(s)\in \mathcal{N}_p(t)\}\in F(I)$$ 
holds for every strong-t-neighborhood $\mathcal{N}_p(t)$ of $p$.
\end{definition}
i.e. $f^{-1}(P \setminus \mathcal{N}_p(t))=\{s\in S:f(s)\notin \mathcal{N}_p(t)\}\in I$ for every strong-$t$-neighborhood.
We use the notation $f\xrightarrow{str-I} p$. If $S=\mathbb{N}$ we obtain the usual definition of strong-$I$-convergence of sequence in PM space. In this case the notation $p_n\xrightarrow{str-I} p$ is used for a real sequence $\{p_n\}$. Now we consider some primary result regarding strong-$I$-convergence for future reference.
\begin{note}\label{20}
(i) If $I$ is an ideal on an arbitrary set  $S$ and let $P$ be PM space then it can be easily verified that Strong-$I$-limit of a function is unique.\\
(ii) If $I_1,I_2$ be ideals on an arbitrary set $S$ such that $I_1\subseteq I_2$ then for each function $f:S\rightarrow P$, we get  $f\xrightarrow{str-I_1} p$ implies $f\xrightarrow{str-I_2} p$.\\
(iii) Again if $P,Q$ are two PM spaces and $g:P\rightarrow Q$ is a continuous mapping and $f:S\rightarrow P$ is strong-$I$-convergent to $p$ then $g\circ f$ is strong-$I$-convergent to $g(x)$.
\end{note}
Since we are working with function, we modify the definition of  strong-$I^*$-convergence in PM space.
\begin{definition}
Let $I$ be an ideal on an arbitrary set $S$ and let $f:S\rightarrow P$ be a function to a PM space $P$. The function $f$ is called strong-$I^*$-convergent to $p \in P$ if there exists a set $M\in F(I)$ such that the function $g:S\rightarrow P$ defined by
\[g(s)=\left\{\begin {array}{ll}
        f(s) & \mbox{if $s\in M$} \\
		p & \mbox{if $s\notin M$}
		\end{array}
		\right. \] \\
is Fin$(S)$-convergent to $p$. 
\end{definition} 
If $f$ is strong-$I^*$-convergent to $p$, then we write $f\xrightarrow{str-I^*} p$. The usual notion of strong-$I^*$-convergence of sequence is a special case for $S=\mathbb{N}$. We write $p_n\xrightarrow{str-I^*} p$ is used for a real sequence $\{p_n\}$.
In the definition of strong-$I^K$-convergence we simply replace the Fin by an ideal on the set $S$. Strong-$I^K$-convergence as a common generalization of all types of strong-$I^*$-convergence of sequences and functions from $S$ to $P$. Here we shall work with functions instead of sequences. One of the reasons is that using functions sometimes helps to simplify notation.
\begin{definition}
Let $K$ and $I$ be an ideal on an arbitrary set $S$, $P$ be a PM space and let $p$ be an element of $P$. The function $f:S\rightarrow P$  is called strong-$I^K$-convergent to $p\in P$ if there exists a set $M\in F(I)$ such that the function $g:S\rightarrow P$ given by
\[g(s)=\left\{\begin {array}{ll}
        f(s) & \mbox{if $s\in M$} \\
		p & \mbox{if $s\notin M$}
		\end{array}
		\right. \] \\
is strong-$K$-convergent to $p$. 
\end{definition} 
\begin{remark}
We can reformulate the definition of strong-$I^K$-convergence in the following way: if there exists an $M\in F(I)$ such that the function $f|_M$ is strong-$K|_M$-convergent to $p$ where $K|_M=\{A\cap M:A\in K\}$.
\end{remark}
If $f$ is strong-$I^K$-convergent to $p$, then we write $f\xrightarrow{str-I^K} p$.
As usual, notion of $I^K$-convergence of sequence is a special case for $S=\mathbb{N}$.

\begin{lemma}\label{0}
If $I$ and $K$ are ideals on an arbitrary set $S$ and $f:S\rightarrow P$ is a function such that $f\xrightarrow{str-K} p$, then $f\xrightarrow{str-I^K} p$.
\end{lemma}
\begin{proof}
The proof is parallel to proof of lemma $3.5$ \cite{23}.
\end{proof}
\begin{proposition}
Let $I,J,K$ and $L$ be ideals on a set $S$ such that $I\subseteq J$ and $K\subseteq L$ and let $P$ be a PM space. Then for any function $f:S\rightarrow P$, we have\\
(i)$f\xrightarrow{str-I^K} p ~~ \Rightarrow ~~ f\xrightarrow{str-J^K} p$ and\\
(ii)$f\xrightarrow{str-I^K} p ~~ \Rightarrow ~~ f\xrightarrow{str-I^L} p$
\end{proposition}
\begin{proof}
(i) Now as $f\xrightarrow{str-I^K} p$ so there exist a set $M\in F(I)$ such that the function $g:S\rightarrow P$ given by
\[g(s)=\left\{\begin {array}{ll}
        f(s) & \mbox{if $s\in M$} \\
		p & \mbox{if $s\notin M$}
		\end{array}
		\right. \]
is strong-$K$-convergent to $p$. Here $M\in F(I)\subseteq F(J)$ as $I\subseteq J$. So obviously $f\xrightarrow{str-J^K} p$.
(ii) The proof directly follows from the fact that $K\subset L$ and the note $\ref{20}(ii)$.
\end{proof}
\begin{theorem}
Let $I$,$K$ be ideals on an arbitrary set $S$, $P$ be a PM space and let $f$ be a function from $S$ to $P$ then \\
(i) $f\xrightarrow{str-I^K} p ~~~ \Rightarrow ~~~ f\xrightarrow{str-I} p$ if $K\subseteq I$. (ii)  $f\xrightarrow{str-I} p ~~~ \Rightarrow ~~~ f\xrightarrow{str-I^K} p$ if $I\subseteq K$.
\end{theorem}
\begin{proof}
(i) Now  $f\xrightarrow{str-I^K} p$, then by the definition of strong-$I^K$-convergence there exist a set $M\in F(I)$ such that the function $g:S\rightarrow P$ given by
\[g(s)=\left\{\begin {array}{ll}
        f(s) & \mbox{if $s\in M$} \\
		p & \mbox{if $s\notin M$}
		\end{array}
		\right. \]
is strong-$K$-convergent to $p$. i.e. $g^{-1}(P\setminus\mathcal{N}_p(t))=f^{-1}(P \setminus \mathcal{N}_p(t))\cap M\in K\subseteq I$ , for every strong-t-neighborhood of $p$. Consequently, $f^{-1}(P \setminus \mathcal{N}_p(t))\subseteq (S\setminus M)\cup g^{-1}(P\setminus\mathcal{N}_p(t)) \in I$ [as $S\setminus M \in I$]. Thus $f\xrightarrow{str-I} p$.\\
(ii) Proof follows from the note $\ref{20}(ii)$ and lemma \ref{0}.
\end{proof}
Now we introduce an example which is strong-$I^K$-convergence but not strong-$I$-convergence.
\begin{example}
Let $K$ and $I$ be two ideals on $S$ such that $K\not\subset I$ and $I\not\subset K$, but $K\cap I\neq \phi$. Consider a set $B\in K\setminus I$. Let $\mathcal{N}_p(t)$ be strong-$t$-neighborhood of $p\in P$ and $q\in P\setminus \mathcal{N}_p(t)$. Let us define the function $f:S\rightarrow P$ by
\[f(s)=\left\{\begin {array}{ll}
        x & \mbox{if $s\in S\setminus B$} \\
		y & \mbox{if $s\in B$}
		\end{array}
		\right. \]
Clearly, $f^{-1}(P \setminus \mathcal{N}_p(t))=B\in K$ so $f\xrightarrow{str-K} p$ and by the lemma \ref{0}, $f\xrightarrow{str-I^K} p$. But  $f^{-1}(P \setminus \mathcal{N}_p(t))=B\notin I$ i.e. $f\overset{str-I}{\nrightarrow} p$.
\end{example}
\begin{note}
Consider the two sets $M_1=\{2n:n\in \mathbb{N}\}$ and $M_2=\{3n:n\in \mathbb{N}\}$ then $2^{M_1}$ and $2^{M_2}$ are two ideals such that $2^{M_1}\not\subset2^{M_2}$ and  $2^{M_2}\not\subset2^{M_1}$ but  $2^{M_1}\cap 2^{M_2}\neq \phi$.
\end{note}
\subsection{Strong-$I^I$ and $(I\vee K)^K$-Convergence}
In this part we discuss strong-$I^K$-convergence for the case when $I=K$ and for any two ideals  $I, K$ on a non-void set $S$, now $I\vee K=\{A\cup B:A\in I, B\in K\}$ is the smallest ideal containing both $I$ and $K$ on $S$ i.e. $I, K\subseteq I\vee K$. It is clear that if $I\vee K$ is non-trivial and $I$ and $K$ are both proper subset of $I\vee K$ then $I$ and $K$ both are non-trivial. But converse part may or may not be true as shown in the following examples.
\begin{example}
Consider the two sets $N_1=\{4n:n\in \mathbb{N}\}$ and $N_2=\{4n-1:n\in \mathbb{N}\}$ now it is clear that $2^{N_1}$, $2^{N_2}$ and $2^{N_1}\vee 2^{N_2}$ all are non-trivial ideal on $\mathbb{N}$.
\end{example}

\begin{example}
Now let  $N_1$ be set of all odd integers and $N_2$ be set of all even integers. Then it is clear that $I=2^{N_1}$, $K=2^{N_2}$ both are non-trivial ideals on $\mathbb{N}$ but $I\vee K$ is a trivial ideal on $\mathbb{N}$.
\end{example}
If $I\vee K$ is a non-trivial on $S$ then the dual filter is $F(I\vee K)=\{G\cap H:G\in F(I),H\in F(K)\}$.

\begin{theorem}
Let $f:S\rightarrow P$ be a map, $I,K$ be ideal on the arbitrary set $S$ and $P$ be a PM space.\\
(i)$f\xrightarrow{str-I} p ~~~ \Leftrightarrow~~~ f\xrightarrow{str-I^I} p$ and\\
(ii)$f\xrightarrow{str-I^K} p ~~~ \Leftrightarrow ~~~ f\xrightarrow{str-(I\vee K)^K} p$
\end{theorem}
\begin{proof}
(i) Proof follows from lemma\ref{0} taking $K=I$.\\
Conversely, let $f$ be strong-$I^I$-convergent to $p$ then there is a set $M\in F(I)$ such that $f|_M$ is strong-$I|_M$-convergent. So for any strong-$t$-neighborhood $\mathcal{N}_p(t)$ of $p$ there exists $G\in F(I)$ such that
$$f^{-1}(\mathcal{N}_p(t))\cap M=G\cap M$$
Clearly $G\cap M\in F(I)$ and $G\cap M\subseteq f^{-1}(\mathcal{N}_p(t))$ i.e. $f^{-1}(\mathcal{N}_p(t))\in F(I)$ i.e. $f$ is strong convergence to $p$.\\
(ii) Suppose that $f$ is strong-$I^K$-convergent to $p$. Then there is a set $M\in F(I)$ such that $f|_M$ is strong-$K|_M$-convergent. Clearly $M\in F(I\vee K)$ since $M\in F(I)$. Therefore $f$ is also strong-$(I\vee K)^K$-convergent to $p$.\\
Conversely, let $f$ is strong-$(I\vee K)^K$-convergent to $p$ i.e. there is a set $M\in F(I\vee K)$ such that $f|_M$ is strong-$K|_M$-convergent. Then for any strong-$t$-neighborhood $\mathcal{N}_p(t)$ of $p$ there exists $G\in F(K)$ such that
$f^{-1}(\mathcal{N}_p(t))\cap M=G\cap M$. Since $M\in F(I\vee K)$, then $M=M_1\cap M_2$ for some $M_1\in F(I)$ and $M_2\in F(K)$. Now we have $$f^{-1}(\mathcal{N}_p(t))\cap M_1\supseteq f^{-1}(\mathcal{N}_p(t))\cap M=(G\cap M_2)\cap M_1$$
Since $G\cap M_2\in F(K)$, this shows that $f^{-1}(\mathcal{N}_p(t))\cap M_1\in F(K|_{M_1})$ i.e. $f$ is strong-$I^K$-convergent to $p$.
\end{proof}
\section{Basic Properties of Strong-$I^K$-Convergence in PM Spaces}
\begin{theorem}
Let $I\vee K$ be a nontrivial ideal on a non empty set $S$ and let $P$ be a PM-space. Then strong-$I^K$-convergence of a function $f:S\rightarrow P$ has a unique strong-$I^K$-limit.
\end{theorem}
\begin{proof}
If possible suppose that the strong-$I^K$-convergent function $f$ has two distinct strong-$I^K$-limits say $p$ and $q$. Since every PM-space is Hausdorff then there exists strong-t-neighborhood $\mathcal{N}_p(t))$ and $\mathcal{N}_q(t))$ for $(t>0)$ such that $\mathcal{N}_p(t)\cap \mathcal{N}_q(t)= \phi$.\\
Now $f$ has strong-$I^K$-limit $p$, so from the definition of strong-$I^K$-limit there exists a set $M_1\in F(I)$ such that the function $g:S \rightarrow P$ given by  
\[g(s)=\left\{\begin {array}{ll}
        f(s) & \mbox{if $s\in M_1$} \\
		p & \mbox{if $s\notin M_1$}
		\end{array}
		\right. \] \\
is strong-$K$-convergent to $p$. So, $g^{-1}(\mathcal{N}_p(t)=\{s\in M_1:g(s)\in \mathcal{N}_p(t)\}\cup \{s\in S\setminus M_1:g(s)\in \mathcal{N}_p(t)\}=(S\setminus M_1)\cup f^{-1}(\mathcal{N}_p(t))=S \setminus (M_1\setminus f^{-1}(\mathcal{N}_p(t)))\in F(K)$
i.e. $M_1\setminus f^{-1}(\mathcal{N}_p(t))\in K$ or $M_1\setminus N_1\in K$ where $N_1= f^{-1}(\mathcal{N}_p(t))$.\\
Similarly, $f$ has strong-$I^K$-limit $q$ so there exists a set $M_2\in F(I)$ s.t. $M_2\setminus f^{-1}(\mathcal{N}_q(t))\in K$ or $M_2\setminus N_2\in K$ where $N_2=f^{-1}(\mathcal{N}_q(t))$
So $(M_1\setminus N_1)\cup (M_2\setminus N_2)\in K$. Then $(M_1\cap M_2)\cap(N_1\cap N_2)^c\subset (M_1\cap N^c_1)\cup(M_2\cap N^c_2)\in K$.
Thus $(M_1\cap M_2)\cap(N_1\cap N_2)^c\in K$ i.e. $(M_1\cap M_2)\setminus (f^{-1}(\mathcal{N}_p(t))\cap f^{-1}(\mathcal{N}_q(t)))\in K$ i.e. 
Since $f^{-1}(\mathcal{N}_p(t)\cap \mathcal{N}_q(t))=\phi$ then $M_1\cap M_2\in K$ i.e.
\begin{equation}\label{6}
S\setminus (M_1\cap M_2)\in F(K)
\end{equation} 
Since $M_1, M_2\in F(I)$,
\begin{equation}\label{7}
M_1\cap M_2\in F(I)
\end{equation} 
 Since $I\vee K$ is non-trivial so the dual filter $F(I\vee K)=\{G\cap H:G\in F(I),H\in F(K)\}$ exists. Now from the equation \ref{6} and equation \ref{7} we get $\phi\in F(I\vee K)$. which is a contradiction. Hence the strong-$I^K$-limit is unique.
\end{proof}
\begin{theorem}
If $I$ and $K$ be two admissible ideal and if there exists an injective function $f:S\rightarrow X\subset P$ which is strong-$I^K$-convergent to $p_0\in P$ then $p_0$ is a limit point of $X$.
\end{theorem}
\begin{proof}
Let the function $f$ has strong-$I^K$-limit $p_0$, so there exists a set $A\in F(I)$ such that the function $g:S \rightarrow P$ given by  
\[g(s)=\left\{\begin {array}{ll}
        f(s) & \mbox{if $s\in A$} \\
		p_0 & \mbox{if $s\notin A$}
		\end{array}
		\right. \] \\
is strong-$K$-convergent to $p_0$. Let $\mathcal{N}_{p_0}(t)$ be an arbitrary strong-$t$-neighborhood. Then the set $C$ (say)$= g^{-1}(\mathcal{N}_{p_0}(t))=\{s:g(s)\in \mathcal{N}_{p_0}(t)\}\in F(K)$. So $C\notin K$ i.e. the set $C$ is an infinite set, as $K$ is admissible ideal. Choose $k_0\in \{s:g(s)\in \mathcal{N}_{p_0}(t)\}$ such that $g(k_0)\neq p_0$ then $g(k_0)\in \mathcal{N}_{p_0}(t)\cap (X\setminus \{p_0\})$. Thus $p_0$ is a limit point of $X$
\end{proof}
\begin{theorem}
A Continuous function $h:P\rightarrow P$ preserves strong-$I^K$-convergence.
\end{theorem}
\begin{proof}
Let the function $f$ has strong-$I^K$-limit $p$, so there exists a set $M\subset S\in F(I)$ s.t. the function $g:S \rightarrow P$ given by  
\[g(s)=\left\{\begin {array}{ll}
        f(s) & \mbox{if $s\in M$} \\
		p & \mbox{if $s\notin M$}
		\end{array}
		\right. \] \\
is strong-$K$-convergent to $p$. Let $\mathcal{N}_p(t)$ be a strong-$t$-neighborhood of point $p$. Then $g^{-1}(\mathcal{N}_p(t))=(S\setminus M)\cup f^{-1}(\mathcal{N}_p(t))=S\setminus (M\setminus f^{-1}(\mathcal{N}_p(t)))\in F(K)$ i.e. $M\setminus f^{-1}(\mathcal{N}_p(t))\in K$. So to prove the theorem, we have to show $h(f(p))\xrightarrow{str-I^K} h(p)$. Now it suffices to show that the function $g_1:S \rightarrow P$ given by
\[g_1(s)=\left\{\begin {array}{ll}
        (h\circ f)(s) & \mbox{if $s\in M$} \\
		h(p) & \mbox{if $s\notin M$}
		\end{array}
		\right. \]
is strong-$K$-convergent to $h(p)$. Let $\mathcal{N}_{h(p)}(t)$ be a strong-$t$-neighborhood containing $h(p)$. Since $h$ is continuous so there exists a strong-$t$-neighborhood $\mathcal{N}_p(t)$ containing $p$ such that $h(\mathcal{N}_p(t))\subset \mathcal{N}_{h(p)}(t)$. Clearly
$\{s: h(f(s))\notin \mathcal{N}_{h(p)}(t)\}\subset \{s:f(s)\notin \mathcal{N}_p(t)\}$ which implies that $\{s:f(s)\in \mathcal{N}_p(t)\}\subset \{s:h\circ f(s)\in \mathcal{N}_{h(p)}(t)\}$ i.e. $f^{-1}(\mathcal{N}_p(t))\subset (h\circ f)^{-1}(\mathcal{N}_{h(p)}(t))$. So $M\setminus (h\circ f)^{-1}(\mathcal{N}_{h(p)}(t))\subset M\setminus f^{-1}(\mathcal{N}_p(t))$. Therefore $M\setminus (h\circ f)^{-1}(\mathcal{N}_{h(x)}(t))\in K$ as $M\setminus f^{-1}(\mathcal{N}_p(t))\in K$. So its complement $g_1^{-1}(\mathcal{N}_{h(p)}(t))\in F(K)$, as required. Hence $h(f(p))\xrightarrow{str-I^K} h(p)$.
\end{proof}
\begin{theorem}
If $P$ has no limit point then strong-$I$-convergence implies strong-$I^K$-convergence, where $I$ and $K$ are two admissible ideals.
\end{theorem}
\begin{proof}
Let $f:S \rightarrow P$ be a function such that $f\xrightarrow{str-I} p$. Since $P$ has no limit point so $\mathcal{N}_p(t)=\{p\}$ is open where $\mathcal{N}_p(t)$ is strong-$t$-neighborhood. Thus we have $f^{-1}(P\setminus \mathcal{N}_p(t))=\{s\in S: f(s)\notin \mathcal{N}_p(t)\}\in I$. Then $M = f^{-1}(\mathcal{N}_p(t))=\{s\in S: f(s)\in \mathcal{N}_p(t)\}\in F(I)$. Thus there exists a set $M\in F(I)$ such that the function $g:S\rightarrow P$ defined by
\[g(s)=\left\{\begin {array}{ll}
        f(s) & \mbox{if $s\in M$} \\
		p & \mbox{if $s\notin M$}
		\end{array}
		\right. \]
is strong-$K$-convergent $p$. (Since for any strong-$t$-neighborhood $\mathcal{N}_p(t)$ containing $p$ then $\{s\in S:g(s)\notin \mathcal{N}_p(t)\}=\phi\in K$). So $f\xrightarrow{str-I^K} p$.
\end{proof}
\begin{note}
Converse of above theorem may not be true. Let $I$ and $K$ be two ideals on a set $S$. Consider a set $A\in K\setminus I$. Let $q\in X\setminus \{p\}$ be a fixed element and define a function $f:S\rightarrow P$ by
\[f(s)=\left\{\begin {array}{ll}
        p & \mbox{if $s\in S\setminus A$} \\
		q & \mbox{otherwise}
		\end{array}
		\right. \]
Now if $\mathcal{N}_p(t)$ is any strong-$t$-neighborhood containing $p$ then $f^{-1}(\mathcal{N}_p(t))=S\setminus A$ if $q\notin \mathcal{N}_p(t)$ and $f^{-1}(\mathcal{N}_p(t))=S$ if $q\in \mathcal{N}_p(t)$. So in both case $f^{-1}(\mathcal{N}_p(t))\in F(K)$. Hence Strong-$K$-$\lim f=p$ then by lemma(\ref{0}) we get Strong-$I^K$-$\lim f=p$. But $\mathcal{N}_p(t_0)=\{p\}$ is also a strong-$t_0$-neighborhood containing $p$ since $P$ has no limit point and $f^{-1}(P\setminus \mathcal{N}_p(t_0)) = A \notin I$. Hence $f$ is not strong-$I$-convergent to $p$.
\end{note}
\subsection{Additive Property with Strong-$I$ and $I^K$-Convergence}
When we are trying to find the relationship between strong-$I$ and $I^K$-convergence, the following condition is important.
Several formulation of AP$(I,K)$-condition are defined in \cite{23}. Before giving the definition we need to state another definition of $K$-pseudo intersection of a system.
\begin{definition} \cite{23}
Let $K$ be an ideal on a set $S$. We write $A\subset_K B$ whenever $A\setminus B\in K.$ If $A\subset_K B$ and $B\subset_K A$ then we write $A\sim_K B$. Clearly $A\sim_K B \Leftrightarrow A\bigtriangleup B\in K$.\\
We say that a set $A$ is $K$-pseudo intersection of a system $\{A_n: n\in \mathbb{N}\}$ if $A\subset_K A_n$ holds for each $n\in \mathbb{N}$
\end{definition}
\begin{definition} \cite{23}
Let $I,K$ be ideals on the set $S$. We say that $I$ has additive property with respect to $K$ or that the condition AP$(I,K)$ holds if any of the equivalent condition of following holds:
\item[(i)] For every sequence $(A_n)_{n\in \mathbb{N}}$ of sets from  $I$ there is $A\in I$ such that $A_n\subset_K A$ for all $n'$s.
\item[(ii)] Any sequence $(F_n)_{n\in \mathbb{N}}$ of sets from $F(I)$ has $K$-pseudo intersection in $F(I)$.
\item[(iii)] For every sequence $(A_n)_{n\in \mathbb{N}}$ of sets from  $I$ there exists a sequence $(B_n)_{n\in \mathbb{N}}\in I$ such that $A_j\sim_K B_j$ for $j\in \mathbb{N}$ and $B=\displaystyle{\cup_{j\in \mathbb{N}}}B_j\in I$.
\item[(iv)] For every sequence of mutually disjoint sets $(A_n)_{n\in \mathbb{N}}\in I$ there exists a sequence $(B_n)_{n\in \mathbb{N}}\in I$ such that $A_j\sim_K B_j$ for $j\in \mathbb{N}$ and $B=\displaystyle{\cup_{j\in \mathbb{N}}}B_j\in I$.
\item[(v)] For every non-decreasing sequence $A_1\subseteq A_2\subseteq \cdots \subseteq A_n\cdots $ of sets from $I$ $\exists$  a sequence $(B_n)_{n\in \mathbb{N}}\in I$ such that $A_j\sim_K B_j$ for $j\in \mathbb{N}$ and $B=\displaystyle{\cup_{j\in \mathbb{N}}}B_j\in I$.
\item[(vi)] In the Boolean algebra $2^S/K$ the ideal $I$ corresponds to a $\sigma$-directed subset,i.e. every countable subset has an upper bound.
\end{definition}
In the case $S=\mathbb{N}$ and $K=$ Fin we get the condition AP from \cite{19} which characterize ideal such that $I^*$-convergence implies $I$-convergence. The condition AP$(I,K)$ is more generalization of condition AP from \cite{6}\cite{19}. \\
\begin{theorem}
Let $I$ and $K$ be two ideals on an arbitrary non-empty set $S$ and $P$ be a PM space. If the ideal $I$ has the additive property with respect to $K$ then strong-$I$-convergence implies strong-$I^K$-convergence.
\end{theorem}
\begin{proof}
Let $f:S\rightarrow P$ be a function such that $f\xrightarrow{str-I} p$. Let $\mathcal{B}=\{\mathcal{N}_p(t_n):n\in \mathbb{N}\}$ be a countable base for $P$ at the point $p$. Now from the definition of strong-$I$-convergence we have $f^{-1}(\mathcal{N}_p(t_n))\in F(I)$ for each n, thus there exists $A\in F(I)$ with $A\subset_K f^{-1}(\mathcal{N}_p(t_n))$ i.e. $A\setminus f^{-1}(\mathcal{N}_p(t_n))\in K$.
Now it suffices to show that the function the $g:S\rightarrow P$ defined by
\[g(s)=\left\{\begin {array}{ll}
        f(s) & \mbox{if $s\in A$} \\
		p & \mbox{if $s\notin A$}
		\end{array}
		\right. \]
is strong-$K$-convergent to $p$. Since for $\mathcal{N}_p(t_n)\in B$ ,we have $g^{-1}(\mathcal{N}_p(t_n))=(S\setminus A)\cup f^{-1}(\mathcal{N}_p(t_n))=S\setminus (A\setminus f^{-1}(\mathcal{N}_p(t_n)))$. Since the set $A\setminus f^{-1}(\mathcal{N}_p(t_n)) \in K$. So $S\setminus (A\setminus f^{-1}(\mathcal{N}_p(t_n)))\in F(K)$ i.e. $g^{-1}(\mathcal{N}_p(t_n))\in F(K)$. Therefore $g$ is strong-$K$-convergent to $p$.
\end{proof}
\section{Strong-$I^K$-Cauchy Functions}
Now we can define in a full generality the notion of Cauchy function and make some basic observations.
\begin{definition}(cf \cite{38})
Let $(P,\mathcal{F},\tau)$ be a PM space. A function $f:S\rightarrow P$ is called strong-$I$-Cauchy if for any $t>0$ there exists an $m\in S$ such that
$$\{s\in S:f(s)\notin \mathcal{N}_{f(m)}(t)\}\in I$$
\end{definition}
\begin{lemma}
Let $(P,\mathcal{F},\tau)$ be a PM space and $I$ be an ideal on a set $S$. For a function $f:S\rightarrow P$ following are equivalent.\\
(i) $f$ is strong-$I$-Cauchy.\\
(ii) For any $t>0$ there is $m\in S$ such that $\{s\in S:f(s)\in \mathcal{N}_{f(m)}(t)\}\in F(I)$.\\
(iii) For every $t>0$ there exists a set $A\in I$ such that  $s,m\notin A$ implies 
$f(s)\in \mathcal{N}_{f(m)}(t)$
\end{lemma}
\begin{proof}
The proof is straightforward and so it is omitted.
\end{proof}
\begin{note}\label{5}
(i) Note that in view of cf \cite{38} in a PM space $(P,\mathcal{F},\tau)$ every strong-$I$-convergent function is strong-$I$-Cauchy.\\
(ii) Clearly if $I_1$, $I_2$ are ideals on a set $S$ such that $I_1\subseteq I_2$ and if $f:S\rightarrow P$ is $I_1$-Cauchy then it is also $I_2$-Cauchy.
\end{note}
\begin{definition}
Let $S$ be an arbitrary set and $(P,\mathcal{F},\tau)$ be a PM space. Let $I,K$ be ideals on the set $S$. A function $f:S\rightarrow P$ is said to be strong-$I^K$-Cauchy if there is $M\in F(I)$ such that the function $f|_M$ is Strong-$K|_M$-Cauchy.
\end{definition}
\begin{note}
If $I$ is an ideal on $S$ and $M\subseteq S$ then we denote by $I|_M$ the trace of the ideal $I$ on the subset $M$ i.e. $I|_M=\{A\cap M:A\in I\}$ and the dual filter is   $F(I|_M)=\{G\cap M: G\in F(I)\}$.
\end{note}
If $K=$Fin we obtain the notion of strong-$I^*$-Cauchy functions. It is relatively easy to see directly from definition and note \ref{5}(ii) that every strong-$I^K$-convergent function is strong-$I^K$-Cauchy.
\begin{lemma}\label{11}
If $I$ and $K$ are ideals on an arbitrary set $S$ and $P$ be a PM space and a function $f:S\rightarrow P$ is strong-$K$-Cauchy then it is also strong-$I^K$-Cauchy.
\end{lemma}
\begin{proof}
If we take $M=S$ then $M\in F(I)$. In this case $K|_M=K$, hence $f$ is strong-$K|_M$-Cauchy. This shows that $f$ is strong-$I^K$-Cauchy.
\end{proof}
\begin{lemma}
Let $I,J,K$ and $L$ be ideals on a set $S$ such that $I\subseteq J$ and $K\subseteq L$ and let $P$ be a PM space. Then for any function $f:S\rightarrow P$, we have\\
(i)strong-$I^K$-Cauchy ~$\Rightarrow$ ~strong-$J^K$-Cauchy and (ii)strong-$I^K$-Cauchy ~$\Rightarrow$ ~strong-$I^L$-Cauchy.
\end{lemma}
\begin{proof}
(i)If $f:S\rightarrow P$ is strong-$I^K$-Cauchy then there is a subset $M\in F(I)$ such that $f|_M$ is strong-$K|_M$-Cauchy. Since $F(I)\subseteq F(J)$, we have $M\in F(J)$. This means that $f$ is also strong-$J^K$-Cauchy.\\
(ii) As $K\subseteq L$ implies $K|_M\subseteq L|_M$. From note \ref{5}(ii) we get that if $f|_M$ is strong-$K|_M$-Cauchy then it is also strong-$L|_M$-Cauchy i.e. $f$ is strong-$I^L$-Cauchy.
\end{proof}
\begin{theorem}
Let $f:S\rightarrow P$ be a map, $I,K$ be ideal on the arbitrary set $S$ and $P$ be a PM space. then\\
(i) $f$ is strong-$I$-Cauchy if and only if it is strong-$I^I$-Cauchy. and\\
(ii) $f$ is strong-$I^K$-Cauchy if and only if it is strong-$(I\vee K)^K$-Cauchy.
\end{theorem}
\begin{proof}
(i) Suppose that $f$ is strong-$I$-Cauchy. Then by lemma \ref{11} it is strong-$I^I$-Cauchy by taking $K=I$.\\
Conversely, let $f$ is strong-$I^I$-Cauchy so there is a set $M\in F(I)$ such that $f|_M$ is strong-$I|_M$-Cauchy. Then for any strong-$t$-neighborhood $\mathcal{N}_f(q)(t)$ of $f(q)$, $q\in S$ the set $C$(say)=$\{p\in S:f(p)\in \mathcal{N}_f(q)(t)\}\cap M\in F(I|_M)$.  So there exists $G\in F(I)$ such that $C=G\cap M$. Clearly $G\cap M\in F(I)$ and $G\cap M\subseteq f^{-1}(\mathcal{N}_{f(q)}(t))$ and so $f^{-1}(\mathcal{N}_{f(q)}(t))\in F(I)$.\\
(ii) Suppose that $f$ is strong-$I^K$-Cauchy. Then there is a set $M\in F(I)$ such that $f|_M$ is strong-$K|_M$-Cauchy. Clearly if $M\in F(I)$ then $M\in F(I\vee K)$. Therefore $f$ is also strong-$(I\vee K)^K$-Cauchy.\\
Conversely, let $f$ is strong-$(I\vee K)^K$-Cauchy. So there is a set $M\in F(I\vee K)$ such that $f|_M$ is strong-$K|_M$-Cauchy. Then for any strong-$t$-neighborhood $\mathcal{N}_{f(q)}(t)$, $q\in S$ there exists $G\in F(K)$ such that $f^{-1}(\mathcal{N}_{f(q)}(t))\cap M=G\cap M$.
Since $M\in F(I\vee K)$, then $M=M_1\cap M_2$ for some $M_1\in F(I)$ and $M_2\in F(K)$. Now we have $$f^{-1}(\mathcal{N}_{f(q)}(t))\cap M_1\supseteq f^{-1}(\mathcal{N}_{f(q)}(t))\cap M=(G\cap M_2)\cap M_1$$
Since $G\cap M_2\in F(K)$, this shows that $f^{-1}(\mathcal{N}_{f(q)}(t))\cap M_1\in F(K|_{M_1})$ i.e. $f$ is strong-$I^K$-Cauchy.
\end{proof}
\section{Strong-$I^K$-Limit Points}
Since we are working with function, we modify the definition of strong-$I$-limit points which is given in \cite{38} in the following way:
\begin{definition}
Let $f:S\rightarrow P$ be a function and $I$  be non-trivial ideal of $S$. Then $q\in P$ is called a strong-$I$-limit point of $f$ if there exists a set $M\subset S$ such that $M\notin I$ and the function $g:S\rightarrow P$ defined by
\[g(s)=\left\{\begin {array}{ll}
        f(s) & \mbox{if $s\in M$} \\
		q & \mbox{if $s\notin M$}
		\end{array}
		\right. \]
is Fin$(S)$-convergent to $q$.
\end{definition}
In the definition of strong-$I^K$-limit point we simply replace the finite ideal by an arbitrary ideal on the set $S$.
\begin{definition}
Let $f:S\rightarrow P$ be a function and $I,K$  be two non-trivial ideals of $S$. Then $q\in P$ is called a strong-$I^K$-limit point of $f$ if there exists a set $M\subset S$ such that $M\notin I,K$ and the function $g:S\rightarrow P$ defined by
\[g(s)=\left\{\begin {array}{ll}
        f(s) & \mbox{if $s\in M$} \\
		q & \mbox{if $s\notin M$}
		\end{array}
		\right. \]
is strong-$K$-convergent to $q$.
\end{definition}
We denote respectively by $I(\Lambda_s(f))$ and $I^K(\Lambda(f))$ the collection of all strong-$I$ and strong-$I^K$-limit points of $f$.
\begin{theorem}
If $K$ is an admissible ideal and $K\subset I$ then $I(\Lambda_s(f))\subset I^K(\Lambda_s(f))$.
\end{theorem}
\begin{proof}
Let $q\in I(\Lambda_s(f))$. Since $q$ is a strong-$I$-limit point of the function $f:S\rightarrow P$, then there exists a set $M\notin I$ such that and the function $g:S\rightarrow P$ defined by
\[g(s)=\left\{\begin {array}{ll}
        f(s) & \mbox{if $s\in M$} \\
		q & \mbox{if $s\notin M$}
		\end{array}
		\right. \]
is Fin$(S)$-convergent to $q$. So for any strong-t-neighborhood $\mathcal{N}_q(t)$ the set $\{s:g(s)\notin \mathcal{N}_q(t)\}\in$ Fin$(S)$. i.e.
$\{s:g(s)\notin \mathcal{N}_q(t)\}$ is a finite set. So $\{s:g(s)\notin \mathcal{N}_q(t)\}\in K$ [as $K$ is n admissible ideal.] Therefore $g$ is strong-$K$-convergent function. Again $M\notin I$ and $K\subset I$ so $M\notin I,K$. Thus $q$ is strong-$I^K$-limit point of $f$ i.e. $q\in I^K(\Lambda_s(f))$. Hence the theorem is proved.
\end{proof}
\begin{theorem}
If every function $f:S\rightarrow P$ has a strong-$I^K$-limit point then every infinite set $A$ in $P$ has an $\omega$-accumulation point when cardinality of $S$ is less or equal to cardinality of $A$.
\end{theorem}
\begin{proof}
Let $A$ be an infinite  set. Define an injective function $f: S\rightarrow A\subset P$. Then $f$ has an strong-$I^K$-limit point say $q$. Then there exists a set $M\subset S$ such that $M\notin I,K$ and the function $g:S\rightarrow P$ given by
\[g(s)=\left\{\begin {array}{ll}
        f(s) & \mbox{if $s\in M$} \\
		q & \mbox{if $s\notin M$}
		\end{array}
		\right. \]
is strong-$K$-convergent to $q$. Let $\mathcal{N}_q(t)$ be a strong-t-neighborhood then $g^{-1}(\mathcal{N}_q(t))=(S\setminus M)\cup f^{-1}(\mathcal{N}_q(t))=S\setminus (M\setminus f^{-1}(\mathcal{N}_q(t)))\in F(K)$ i.e. $M\setminus f^{-1}(\mathcal{N}_q(t))\in K$. So $f^{-1}(\mathcal{N}_q(t))\notin K$.[For if $f^{-1}(\mathcal{N}_q(t))\in K$ then we get $M\in K$, which is a contradiction.] So $\{s:f(s)\in \mathcal{N}_q(t)\}$ is an infinite set. Consequently $\mathcal{N}_q(t)$ contains infinitely many points of the function $f(s)$ in $P$. So $\mathcal{N}_q(t)$ contains infinitely many elements of $A$. Thus $q$ becomes $\omega$-accumulation point of $A$.
\end{proof}
\section{Acknowledgment}
The second author is grateful to the University of Burdwan, W.B., India for providing State Fund Fellowship during the preparation of this work.


\begin{thebibliography}{9}
\bibitem{1} M. Balcerzak and K. Dems and A. Komisarski, \textit{Statistical convergence and ideal convergence for sequences of functions}, J. Math Anal. Appl. \textbf{328(1)} (2007), 715-729.
\bibitem{2} A.K. Banerjee and A. Banerjee, A Note on $I$-Convergence and $I^*$-Convergence of Sequence and Nets in Topological Spaces, \textit{Mathematicki Vesnik} \textbf{67(3)} (2015), 212-221.

\bibitem{3} A.K. Banerjee and A. Banerjee, $I$-convergence classes of sequences and nets in topological spaces, \textit{Jordan Journal of Mathematics and Statistics} \textbf{11(1)} (2018), 13-31.

\bibitem{3.1} A.K. Banerjee and R. Mondal, A note on convergence of double sequences in topological spaces, \textit{Mathematicki Vesnik} \textbf{69(2)} (2017), 144-152.

\bibitem{3.2} A.K. Banerjee and A. Dey, Metric Spaces and Complex Analysis, \textit{New Age International Publishers}(2008).

\bibitem{3.3} A.K. Banerjee and A. Banerjee, A Study on $I$-Cauchy sequences and $I$-divergence in $S$-metric Spaces, \textit{Malaya Journal of Mathematik} \textbf{6} (2018) 226-230.

\bibitem{3.4} A.K. Banerjee and A. Banerjee, $I$-completeness in funtion spaces, \textit{arXive: 1704.05279v1}(2017)

\bibitem{3.5} A.K. Banerjee and M. Paul, A Note on $I^K$ and $I^{K^*}$-Convergence in Topological Spaces, \textit{arXive: 1807.11772v1}(2018)

\bibitem{4} Katarzyna Dems, \textit{On $I$\text -cauchy sequences}, Real Analysis Exchange, \textbf{30(1)} (2004/2005), 123-128.
\bibitem{6} P. Das, P. Kostyrko, P. Malik, W. Wilczy\' nski, \textit{$I$ and $I^*$\text -convergence of double sequences}, Math. Slov. \textbf{58(5)} (2008), 605-620.
\bibitem{7} P. Das, M. Sleziak, V. Toma, \textit{$I^K$-Cauchy functions}, Topology and its Applications \textbf{173} (2014), 9-27.
\bibitem{8} P. Das and S. K.Ghosal, \textit{Some further results on I-Cauchy sequences and Condition(AP)}, Comput. Math. Appl. \textbf{59(8)} (2010), 2597-2600.
\bibitem{10} H. Fast, \textit{Sur la convergence ststistique}, Colloq. Math. \textbf{2} (1951), 241-244.
\bibitem{12} J.A. Fridy, \textit{On statistical convergence}, Analysis \textbf{11} (1991) 59-66.
\bibitem{17} B.K. Lahiri, P. Das, \textit{$I$ and $I^*$-convergence in topological spaces}, Math. Bohem. \textbf{2} (2005), 153-160.
\bibitem{18} P. Kostyrko, M. Ma\u caj, T.\u Sal\' at, \textit{Statistical convergence and $I$-convergence}, Unpublished, http://thales.doa.fmph.uniba.sk/macaj/ICON.pdf.\bibitem{19} Pavel Kostyrko,Tibor \u Sal\' at, Wladyslaw Wilczy\' nski, \textit{$I$-convergence}, Real Analysis Exchange \textbf{26(2)} (2000/2001) ,669-686.
\bibitem{20} M. Macaj, T. Sal\' at, \textit{Statistical convergence of subsequence of a given sequence}, Math. Bohem \textbf{126} (2001), 191-208.
\bibitem{23} M. Macaj, M. Sleziak, \textit{$I^K$-convergence}, Real Analysis Exchange \textbf{36(1)} (2010/2011), 177-194.
\bibitem{24} K. Menger, \textit{Statistical metrics}, Proceedings of the National Academy of Sciences of the United States of America \textbf{28} (1942), 535-537.
\bibitem{27} Mursaleen, $\lambda$\textit{-statistical convergence}, Mathematica Slovaca \textbf{50(1)} (2000), 111-115.
\bibitem{28} M. Mursaleen and S. A. Mohiuddine, \textit{On ideal convergence in probabilistic normed spaces}, Mathematica Slovaca \textbf{62(1)} (2012), 49-62.
\bibitem{29} Anar Nabiev, Serpil Pehlivan and Mehmet G\" urdal, \textit{On $I$\text -cauchy sequence}, Taiwanese Journal of Mathematics \textbf{11(2)} (2007), 569-576.
\bibitem{31} D. Rath, B.C. Tripathy, \textit{On statistically convergent and statistically Cauchy sequence}, Indian J. Pure Appl. Math. \textbf{25(4)} (1994) 381-386.
\bibitem{32} Tibor \u Sal\' at, \textit{On statistically convergent sequence of real numbers}, Mathematica Slovaca \textbf{30(2)} (1980), 139-150.
\bibitem{33} H. Steinhaus, \textit{Sur la convergence ordinaire et la convergence asymptotique}, Colloq. Math. \textbf{2} (1951) 73-74.
\bibitem{35} B. Schweizer and A. Sklar, \textit{Probabilistic Metric Spaces}, New York: Elsevier Science Publishing Co. (1983)
\bibitem{36} B. Schweizer and A. Sklar, \textit{Statistical metric spaces}, Pacific Journal of Mathematics \textbf{10}, 313-334.
\bibitem{37} C. Sencimen and S. Pehlivan, \textit{Strong statistical convergence in probabilistic metric spaces}, Stochastic Analysis and Applications \textbf{26(3)} (2008), 651-664.
\bibitem{38} C. Sencimen and S. Pehlivan, \textit{Strong ideal convergence in probabilistic metric spaces}, Indian Academy of Sciences \textbf{119(3)} (2009), 401-410.
\bibitem{38.1} A. N. Šerstnev, \textit{Random normed spaces: Problems of completeness}, Kazan. Gos. Univ.Učen. Zap. \textbf{122} (1962), 3-20.
\bibitem{39} R. M. Tardiff, \textit{Topologies for probabilistic metric spaces}, Pacific Journal of Mathematics \textbf{65(1)} (1976), 233-251. 
\end{thebibliography}
\end{document}